\documentclass[11pt]{amsart}
\usepackage{amsmath,amssymb}
\usepackage[hyperindex]{hyperref}
\usepackage{amsrefs,booktabs,verbatim}
\usepackage{xcolor}

\title{Aspects of stable polynomials}
\author{Steve Fisk}
\date{\today}

\newtheorem{fact}{Fact}

\newtheorem{question}{Question}
\theoremstyle{definition}
\newtheorem{definition}{Definition}
\newtheorem{remark}{Remark}
\newtheorem{example}{Example}

\newenvironment{aside}{}{}

\newcommand{\notimplies}{\hspace*{.3cm}\not\hspace*{-.3cm}\implies}

\newcommand{\rhpc}{\overline{\textsf{RHP}}}
\newcommand{\rhp}{\textsf{RHP}}
\newcommand{\uhp}{\textsf{UHP}}

\newcommand{\pposlace}{\ensuremath{\stackrel{P}{\thicksim}}}
\newcommand{\hposlace}{\ensuremath{\stackrel{H}{\thicksim}}}
\newcommand{\hlace}{\ensuremath{\stackrel{H}{\longleftarrow}}}
\newcommand{\place}{\ensuremath{\stackrel{P}{\longleftarrow}}}
\newcommand{\ulace}{\ensuremath{\stackrel{U}{\longleftarrow}}}

\newcommand{\stabled}[1]{{\mathcal{H}_{#1}}} 
\newcommand{\stabledc}[1]{{\mathcal{H}_{#1}(\complexes)}} 
\newcommand{\stabledcbar}[1]{{\widehat{\mathcal{H}}_{#1}(\complexes)}} 

\newcommand{\polypos}[1]{P^{+}_{#1}}

\newcommand{\complexes}{{\mathbb C}}
\newcommand{\imag}{\boldsymbol{\imath}}
\newcommand{\reals}{{\mathbb R}}
\newcommand{\up}[1]{\text{\texttt{U}}_{#1}\,(\complexes)}
\newcommand{\rup}[1]{{\text{\texttt{U}}_{#1}}}
\newcommand{\sdiffi}{{\text{{\textbf{\tiny{I}}}}}} 

\newcommand{\xx}{\mathbf{x}}
\newcommand{\yy}{\mathbf{y}}
\newcommand{\uu}{\mathbf{u}}

\newcommand{\smalltwodet}[4]{%
\left|\begin{smallmatrix} #1&#2 \\ #3&#4\end{smallmatrix}\right|}

\begin{document}
\maketitle

\begin{aside}
This article is an introduction to the properties of stable
polynomials in several variables with real or complex
coefficients. These polynomials are defined in terms of where the
polynomial is non-vanishing. 
\end{aside}

\begin{definition}
  $\stabledc{d} = \left\{\text{\parbox{3.5in}{All polynomials
      $f(x_1,\dots,x_d)$ with complex coefficients such that
      $f(\sigma_1,\dots,\sigma_d)\ne0$ for all
      $\sigma_1,\dots,\sigma_d$ in the right half plane. If we don't
      need to specify $d$ we simply write $\stabledc{}$.}}\right.$ We call
such polynomials  \emph{stable polynomials}. In one variable they are
often called Hurwitz stable.

\end{definition}

\begin{aside}
This article is a companion to \cite{fisk-upper} where we studied polynomials
non-vanishing in the upper half plane.
\end{aside}

\begin{definition}
  $\up{d} = \left\{\text{\parbox{3.5in}{All polynomials
      $f(x_1,\dots,x_d)$ with complex coefficients such that
      $f(\sigma_1,\dots,\sigma_d)\ne0$ for all
      $\sigma_1,\dots,\sigma_d$ in the upper half plane. If we don't
      need to specify $d$ we simply write $\up{}$.}}\right.$ 
Such polynomials are called \emph{upper polynomials}.

\end{definition}

\begin{aside}
Since multiplication by $\imag$ takes the right half plane to the
upper half plane, we have the important fact that
\[
f(x_1,\dots,x_d)\in\up{d} \Longleftrightarrow 
f(\imag x_1,\dots,\imag x_d)\in\stabledc{d}
\]
Consequently, many of the properties of $\stabledc{}$ follow
immediately from the properties of $\up{}$, and their proofs will be
omitted. See \cites{bbs,bb2,bbl,bbs-johnson,branden-hpp,fisk,fisk-upper,wagner}
for more information about $\up{}$ and $\stabledc{}$. We use the
notation in \cite{fisk-upper}, and let $\rhp$ denote the right half plane.
\end{aside}

\begin{aside}
  We do not cover  well-known topics in one variable such as
  Routh-Hurwitz, the Edge theorem, and Kharitonov theory. See
  \cites{rahman,kharitonov}. 
\end{aside}

\section{Complex coefficients}
\label{sec:complex-coefficients}

\begin{aside}
These results are proved  following the corresponding arguments for
upper polynomials.  The main difference is the disappearance of minus
signs for $\stabledc{}$.
\end{aside}
  
\begin{fact}\label{lots elem}\ 
  \begin{enumerate}
  \item   Suppose $f(\xx)\in\stabledc{d}$.
  \begin{enumerate}
  \item If $\alpha\ne0$ then $\alpha f(\xx)\in\stabledc{d}$.  
  \item If $a_1>0,\dots,a_d>0$ then $f(a_1x_1,\dots,a_dx_d)\in\stabledc{d}$.
    \item If $\Re(\sigma_1)>0,\dots,\Re(\sigma_d)>0$ then
      $f(x_1+\sigma_1,\dots,x_d+\sigma_d)\in\stabledc{d}$.
    \item If $\Re(\sigma)>0$ then $f(\sigma,x_2,\dots,x_d)\in\stabledc{}$.
    \item $f(x_1+y,x_2,\dots,x_d)\in\stabledc{}$.
    \item $f(x,x,x_3,\dots,x_d)\in\stabledc{}$.
  \end{enumerate}
\item If $f(\xx)\in\stabledc{}$ then $f(\imag a,x_2,\dots,x_d)\in\stabledc{}\cup\{0\}$ if $a\in\reals$.
\item   $\stabledc{}$ is closed under multiplication and extracting factors.
\item   If $\sum_0^n f_i(\xx)\, y^i \in\stabledc{}$ then $\sum_0^n f_i(\xx)\,y^{n-i}\in\stabledc{}$.
\item If $f(\xx)\in\stabledc{}$ then  $\partial_{x_i}f(\xx)\in\stabledc{}\cup\{0\}$.
\item If $\sum f_i(\xx) y^i \in\stabledc{}$ then all coefficients $f_i(\xx)$
are in $\stabledc{}\cup\{0\}$.

  \end{enumerate}
  
\end{fact}

\begin{aside}
   We have interlacing for both $\stabledc{}$ and $\up{}$; later we will
 see there is another kind of interlacing when all the coefficients are
 real. Note that interlacing is symmetric for stable polynomials.
\end{aside}

\begin{definition}

\[  \begin{matrix}
    f(\xx) &\hlace &g(\xx)&  \quad \text{if and only if} \quad &
    f(\xx)+y\, g(\xx)  \in & \stabledc{} \\
    f(\xx) &\ulace &g(\xx)&  \quad \text{if and only if} \quad &
    f(\xx)+y\, g(\xx)  \in &\up{} 
  \end{matrix}
\]
\end{definition}

\begin{fact}\label{elem-2}\
  \begin{enumerate}
  \item If $\sum_0^n
  f_i(\xx)y^i\in\stabledc{}$ then $f_i(\xx) \hlace f_{i+1}(\xx)$ for 
  $i=0,\dots,n-1$, provided $f_i$ and $f_{i+1}$ are not both zero.
\item If $f\in\stabledc{}$ then $f\hlace \partial_{x_i}f(\xx)$.
\item Suppose $f,g\in\stabledc{}$
  \begin{enumerate}
  \item $fg\hlace gh$ iff $h\in\stabledc{}$ and $f\hlace g$.
  \item If $f\hlace g$ then $g \hlace f$.
  \item If $f\hlace g$ and $f\hlace h$ then
    $f\hlace g+h$.
  \item If $f\hlace g$ and $h\hlace g$ then
    $f+h\hlace g$.
  \item If $f\hlace g\hlace h$ then $f+h\hlace g$.
  \end{enumerate}
\item (Hermite-Biehler) Suppose that $f(\xx)$ is a polynomial, and
  write $f(\xx) = f_e(\xx)+f_o(\xx)$ where $f_e(\xx)$
  (resp. $f_o(\xx)$) consists of all terms of even (resp. odd)
  degree. Then
  \begin{quote}
    $f\in\stabledc{}$ if and only if $f_e\hlace f_o$.
  \end{quote}

  \end{enumerate}
  
\end{fact}

\begin{definition}
  If $f(\xx)$ is a polynomial then $f^H(\xx)$ is the sum  of all terms
  with highest total degree. 
\end{definition}

\begin{fact}\label{homog}
Suppose that $f(\xx)\in\stabledc{d}$ is an stable polynomial of
  degree $n$.
  \begin{enumerate}
  \item $f^H(\xx)$ is a stable homogeneous polynomial.
  \item $f^H$ is the limit of stable homogeneous  polynomials such that
    all monomials of degree $n$ have non-zero coefficient.
  \item All the coefficients of $f^H$ have the same argument.
  \end{enumerate}  
\end{fact}

\section{Real coefficients}
\label{sec:real-coefficients}

\begin{definition}
$\stabled{d}$ is the subset of $\stabledc{d}$ whose coefficients are
all positive.  
\end{definition}

\begin{aside}
  The restriction that the coefficients are all positive is natural:
\end{aside}

\begin{fact}
  If $f(\xx)\in\stabledc{}$ has all real coefficients then all coefficients have the
  same sign.
\end{fact}
\begin{proof}
  Suppose $f(\xx)$ has degree $n$.  We can write $f(\xx)$ as a limit
  of $f_\epsilon(\xx)$ where $f_\epsilon$ is stable, and all
  coefficients of terms of degree at most $n$ are non-zero. Since the
  signs of a stable polynomial in one variable are all the same, an
  easy induction shows that all coefficients of $f_\epsilon$ have the
  same sign. Taking limits finishes the proof.
\end{proof}

\begin{fact}
  If $a,b,c,d$ are positive  then
  $a+bx+cy + dxy\in\stabled{2}$.
\end{fact}
\begin{proof}
If $f(x,y)=a+bx + cy + dxy=0$ where $x=r+\imag s$ then 
\[
\Re(y) = -\frac{(a c + b c r + a d r + b d r^2 + b d s^2)}{(c + d r)^2
  + d^2 s^2}
\]
If $r+\imag s\in\rhp$ then $r>0$ so $\Re(y)\not\in\rhp$. Thus $f(x,y)$
can not have both $x,y$ in the right half plane, so $f$ is stable.
\end{proof}

\begin{aside}
  There are constructions of polynomials in $\stabled{}$ and
  $\stabledc{}$ using determinants, and they follow from the
  corresponding results for $\rup{}$.
\end{aside}

\begin{fact}  Suppose that $A$ is skew-symmetric, $S$ is symmetric,
  and  all  $D_i$ are positive definite.
\begin{align*}
    \bigl| I + x_1\,D_1 + \cdots + x_d\,D_d + \imag S\bigr | &
    \in\stabledc{d}  \\
   \bigl| I + x\,D_1 + \cdots + x_d\,D_d +  A\bigr | & \in\stabled{d} 
 \end{align*}
\end{fact}

\begin{aside}
  Next we introduce polynomials that share the properties of both
  stable and upper polynomials.
\end{aside}

\begin{definition} Let $\polypos{d} = \stabled{d}\cap\rup{d}$.
  Interlacing in $\polypos{d}$ is defined as usual
  \begin{align*}
f\place g  & \text{ iff } f+yg\in \polypos{d}
\end{align*}
{and is equivalent to}
$  f\hlace g \text{ and } f\ulace g$

\end{definition}

\begin{aside}
  In one variable, $\polypos{1}$ consists of all polynomials with all
  real roots,  no positive roots, and all positive
  coefficients.  Here's an example of a polynomial with positive
  coefficients that is not in
  $\polypos{2}$. Let $f(x,y) = x(x+3)+y\,(x+1)(x+2)$. The roots of
  $f(x,1+\imag)$ lie in the second and third quadrants, so
  $f(x,1+\imag)\not\in\up{2}$, and hence is not in $\polypos{2}$.
\end{aside}

\section{Analytic closure}
\label{sec:analytic-closure}

\begin{definition}
  $\stabledcbar{d}$ is the uniform closure on compact subsets of $\stabledc{d}$.
\end{definition}

\begin{fact}\label{exy} \
\begin{enumerate}
\item  $e^{\xx\cdot\yy} \in\stabledcbar{2d}$ and
  $e^{\xx\cdot\xx}\in\stabledcbar{d}$.
\item If $f(\xx,\yy)\in\stabledc{2d}$ then $e^{\partial_\xx\cdot\partial_\yy}f(\xx,\yy)\in\stabledc{2d}$.
\item   If $f(\xx)\in\stabledc{d}$ and $g(\xx)\in\stabledc{d}$ then
  $f(\partial_\xx)g(\xx)\in\stabledc{d}\cup\{0\}$. This also holds for $f\in\stabledcbar{d}$.
\item   Suppose $T$ is a non-trivial linear transformation defined on polynomials in $d$
  variables.  $T(e^{\xx\cdot\yy})\in\stabledcbar{2d}$ if and only if $T$ maps
  $\stabledc{d}\cup\{0\}$ to itself.

\end{enumerate}
\end{fact}

\begin{example} The Hadamard product of $f(y)$ and $g(\xx,y)$ is given by
\[ f\ast g = \bigl(\sum a_i\,y^i\bigr)\ast \bigl(\sum g_i(\xx)\,y^i
\bigr) = \sum a_i\, g_i(\xx)\, y^i \]
  \begin{enumerate}
  \item It is easy to verify that $f(x)\in\polypos{1}$ then $f(xy)\in\stabledc{2}$.
  \item The linear transformation $\exp\colon x^i \mapsto x^i/i!$
    determines a map \\  $\stabledc{1}\longrightarrow\stabledc{1}$. This
    follows from the above since the generating function $f(x,y)$ of $\exp$ is
    $g(xy)$ where $g = \sum x^i/(i!i!)\in\polypos{1}$.
  \item The linear transformation $g\mapsto f\ast g$ has
    generating function 
\[
\sum ( f(y)\ast y^i\xx^\sdiffi) \frac{v^i \uu^\sdiffi}{i! \sdiffi!} =
e^{\xx\cdot\uu} \exp f(yv)
\]
This is in $\stabledcbar{}$, so we have a map
$\polypos{1}\times\stabledc{d} \longrightarrow \stabledc{d}$.
It is surprising \cite{garloff-wagner} that the Hadamard also maps
$\stabled{1}\times\stabled{1}\longrightarrow \stabled{1}$.

  \end{enumerate}

\end{example}

\begin{fact}\label{f(x,D)}
  Suppose that $f(\xx,\yy)$ is a polynomial, and define $T(g) =
  f(\xx,\partial_\yy)g$. The following are equivalent.
  \begin{enumerate}
  \item $T:\stabledc{d}\longrightarrow \stabledc{d}\cup{0}$.
  \item $f(\xx,\yy)\in\stabledc{2d}$.
  \end{enumerate}
\end{fact}

\section{Positive interlacing}
\label{sec:positive-interlacing}

If $f\hlace g$ then $f+\sigma g$ is stable for all $\sigma$ in the
right half plane. If we restrict ourselves to just the real numbers in
the right half plane we get \emph{positive interlacing}.

\begin{definition}
\[  \begin{matrix}
    f(\xx) &\pposlace & g(\xx)&  \quad \text{if and only if} \quad &
    f(\xx)+ r\, g(\xx)  \in & \polypos{} & \text{ for all $r>0$} \\
    f(\xx) &\hposlace & g(\xx)&  \quad \text{if and only if} \quad &
    f(\xx)+r\, g(\xx)  \in &\stabled{} & \text{ for all  $r>0$} \\
  \end{matrix}
\]
\end{definition}

\begin{aside}
  Clearly both $\hposlace$ and $\pposlace$ are reflexive. In one
  variable $f\pposlace g$ is equivalent to common interlacing. That
  is, $f\pposlace g$ if and only if there is an $h\in\polypos{1}$ so
  that $h\place f$ and $h\place g$. In general it follows from
  Fact~\ref{elem-2} that
\end{aside}

\begin{fact}\label{pos lace 1}\
  \begin{enumerate}
  \item If $f,g\in\polypos{}$ and there is an $h$ such that $h\place
    f$ and $h\place g$ then $f\pposlace g$.
  \item If $f\in\polypos{}$ then $\partial_{x_i}(f)\pposlace \partial_{x_j}(f)$.
  \end{enumerate}
  
\end{fact}

Here are some elementary properties of $\hposlace$ and  $\pposlace$.

\begin{fact}Suppose $r,s$ are positive.

\begin{tabular}{lcl}
(1) $f\hposlace g$ iff $r f \hposlace s g$. &\qquad& 
(2) $f\hposlace g$ iff $f + r g \hposlace g$ \\
(3) If $h\in\stabled{}$ then $f\hposlace g$ iff $fh \hposlace hg$. &&
(4) $f\hposlace g$ iff $g \hposlace f$. \\
(5) $f\pposlace g \implies f\hposlace g$ &&
(6) $f\hposlace g \notimplies f\pposlace g$ \\
(7) $f\place g \implies f\pposlace g$ &&
(8) $f\pposlace g\notimplies f\place g$\\
(9) $f\hlace g \implies f\hposlace g$ &&
(10) $f\hposlace g\notimplies f\hlace g$
\end{tabular}
\end{fact}
\begin{proof}
  The only ones that need any proof are the implication failures.
  \begin{description}
\item[$f\hposlace g \notimplies f\place g$] If $f=x^2$, $g=1$ and $t>0$
then $f+tg = x^2+t\in\stabled{}$, but it is easy to check that
$x^2+y\not\in\polypos{2}$.
  \item[$f\hposlace g\notimplies f\hlace g$] The above example also
    shows this, since $x^2+y\not\in\stabled{2}$.
\item[$f\pposlace g\notimplies f\place g$] If we let $f=x(x+3)$ and $g
  = (x+1)(x+2)$ then $f\pposlace g$. However, we have seen that
  $f+yg\not\in\polypos{2}$. 

  \end{description}
\end{proof}

\begin{fact}\ \label{fact:11}
  \begin{enumerate}
  \item If $\sum f_i(\xx)y^i\in\stabled{}$ then $f_k(\xx)\hposlace
    f_{k+2}(\xx) $ for $k=0,1,\dots$.
  \item If $f(\xx) + \cdots + yz\,g(\xx) + \cdots \in \stabled{}$ then
    $f\hposlace g$.
   \item If $f\hlace g$ then $f\hposlace xg$.
  \item If $f\hlace g$ and $f_1\hlace g_1$ then $ff_1 \hposlace gg_1$.
  \end{enumerate}
\end{fact}
\begin{proof}
  Since $y^2+t\in\stabled{}$ for positive $t$, we know
  $\partial_y^2+t$ preserves $\stabled{}$. We first differentiate $k$
  times, yielding a polynomial in $\stabled{}$.  The constant term of
\begin{multline*}
(\partial_y^2 + t)\bigl[k!f_k(\xx) + (k+1)!f_{k+1}(\xx)y + (1/2)(k+2)! f_{k+2}(\xx) y^2 + \cdots\bigr]
\\ = 
\bigl[t k! f_k(\xx) + (k+2)! f_{k+2}(\xx)\bigr] + y(\cdots)
\end{multline*}
is in $\stabled{}$, which proves the first part. For the second part
we multiply by $yz+t$ where $t\ge0$ and consider the coefficient of $yz$.

The third part follows from first part
and the expansion
\[
(f+yg)(xy+1) = f + y(g+xf) + xg\,y^2
\]

Next, we compute
\[
(f+yg)(f_1+yg_1) = ff_1 + y(fg_1+f_1g) + y^2gg_1
\]
and the conclusion follows as above.
\end{proof}

\section{Properties of one variable}
\label{sec:prop-one-vari}

\begin{aside}
  In this section we assume $d=1$.  When we restrict ourselves to one
  variable there are some useful techniques for showing interlacing
  holds. We let $Q_i$ denote the $i$'th quadrant.
\end{aside}

\begin{fact}\label{one variable}\ 
  \begin{enumerate}
  \item \makebox[1.5cm]{$f\hposlace g$} iff $\frac{f}{g}\colon Q_1\longrightarrow \complexes\setminus(-\infty,0)$.
  \item \makebox[1.5cm]{$f\hlace g$} iff $\frac{f}{g}\colon Q_1\longrightarrow \rhpc$.
  \item  \makebox[1.5cm]{$f\place g$} iff $\frac{f}{g}\colon Q_1\longrightarrow Q_1$.
  \item  \makebox[1.5cm]{$f\pposlace g$} iff $\frac{f}{g}\colon Q_1\longrightarrow \rhp$
  \item \makebox[1.5cm]{$f\pposlace g$} implies $f\hlace g$.
  \end{enumerate}
\end{fact}
\begin{proof}
  $f\hposlace g$ iff 
  $f + tg \in\stabled{1}$ for $t>0$ iff
  $-\frac{f}{g}(\sigma)\not\in (0,\infty)$ for $\sigma\in\rhp$ iff 
  $\frac{f}{g}\colon \rhp\longrightarrow
  \complexes\setminus(-\infty,0)$ iff 
  $\frac{f}{g}\colon  Q_1\longrightarrow \complexes\setminus(-\infty,0)$ since $f$ and $g$
  have real coefficients.

  Next, $f\hlace g$ iff $f+yg\in\stabled{2}$ iff
  $-\frac{f}{g}(\sigma)\not\in\rhp$ for $\chi\in\rhp$ iff
  $\frac{f}{g}\colon \rhp\longrightarrow\rhpc$ iff 
  $\frac{f}{g}\colon Q_1\longrightarrow\rhpc$.

  If $f\place g$ then we take $f$ to be monic, and we can write
  $f=\prod(x+r_i)$ and $g = \sum a_i f/(x+r_i)$ where $r_i$ and $a_i$
  are non-negative. If $\sigma\in Q_1$ then $(g/f)(\sigma) =
  \sum a_i/(\sigma+r_i) \in Q_4$, and so $(f/g)(\sigma)\in Q_1$. 

  If $f+tg\in\polypos{1}$ for all positive $t$ then $f$ and $g$ have a
  common interlacing\cite{fisk}, so there is an $h$ satisfying $h\longleftarrow
  f$ and $h\longleftarrow g$. If $\sigma\in Q_1$ then
  $(f/g)(\sigma) = \frac{f}{h}\frac{h}{g}(\sigma)$. By the first
  part $(h/f)(\sigma)\in Q_1$ and $(g/h)(\sigma)\in Q_4$, so
  $(f/g)(\sigma)\in\rhp$. 

  The last one follows from the  previous ones.

  \end{proof}

\begin{fact}\label{positive interlace} \ 
  \begin{enumerate}
  \item If $f\pposlace f_1$,\dots,$f\pposlace f_n$ then $f\pposlace f_1+\cdots+f_n$.
  \item If $f\place f_i$ and $g\place g_i$ for $1\le i\le n$ then $fg
    \hposlace \sum f_i g_i$.
  \item If $f\place g$ then $\smalltwodet{f}{g}{f'}{g'}$ is stable.
  \item If $f\place g$ then the Bezout polynomial 
    $B(x,y)=\frac{1}{x-y}\smalltwodet{f(x)}{f(y)}{g(x)}{g(y)}$ is
    stable.
  \end{enumerate}
\end{fact}
\begin{proof}
  The first one follows from Fact~\ref{one variable}(4) since
  $\frac{f_1+\cdots+f_n}{f} = \frac{f_1}{f}+\cdots+\frac{f_n}{f}$. The
  second follows from 
\[ \frac{1}{fg}\sum f_ig_i = \sum \frac{f_i}{f}\,\frac{g_i}{g}
\in\uhp\]
by Fact~\ref{one variable}(3).  For the
  next one, write $f=\prod(x+r_i)$ and $g = \sum a_i \frac{f}{x+r_i}$
  where the $a_i$ are non-negative then
  \[ (f'g-fg')(\sigma) = f^2\,\sum \frac{a_i}{(\sigma+r_i )^2} \] and
   this is in the lower half plane, so $f^2 \hlace f'g-g'h$. In
   particular, $f'g-g'h$ is stable.

   Finally, we show that $f(x)f(y)\hposlace B(x,y)$ by using the identity
\[ B(x,y) = \sum a_i \frac{f(x)}{x-r_i}\frac{f(y)}{y-r_i}
\]
\end{proof}

\begin{example}
  Assume that $f_0,f_1,f_2,\dots$ is an orthogonal polynomial
  sequence.  The Christoffel-Darboux formula states that
\[
 f_0(x)^2 + f_1(x)^2+\cdots+f_n(x)^2 =
  \frac{k_n}{k_{n+1}}
\begin{vmatrix} f_{n}(x) & f_{n+1}(x) \\ f_{n}'(x) &
  f_{n+1}'(x)\end{vmatrix}\\
\]
for certain constants $k_n$.  It follows from  Fact~\ref{positive interlace} that

\[
 f_0(x)^2 + f_1(x)^2+\cdots+f_n(x)^2 \quad \text{is stable.}
\]

\end{example}

We can generalize the third part of Fact~\ref{positive interlace} in
several ways. The proofs are similar to Fact~\ref{positive interlace}
-- write the desired determinant in terms of the parameters and simple
geometry implies it satisfies the conditions for positive interlacing.

\begin{fact}\ 
  \begin{enumerate}
  \item If $\sum f_i(x)y^i \in\polypos{2}$ and $0<\alpha<2$ then $(f_0)^2 \hposlace
    \smalltwodet{f_1}{\alpha f_0}{f_2}{f_1}$.
  \item If $\sum f_i(x)y^i \in\polypos{2}$ then
    $\smalltwodet{f_k}{f_{k+1}}{f_{k+1}}{f_{k+2}}$ is stable. 
  \item Suppose that $D_1,D_2,D_3$ are positive definite matrices,
    where $D_1$ is diagonal, and $D_2,D_3$ have all positive entries. If
\[ \bigl| Id + x D_1 + yD_2 + zD_3\bigr| = f(x) + y\,g(x)+ z\,h(x) + yz\,k(x) +
\cdots \]
then $\smalltwodet{f}{g}{h}{k}$ is stable.
  \end{enumerate}
  
\end{fact}

\section{Questions}
\label{sec:questions}

\begin{question}
  If $f\pposlace g$ or $f\hposlace g$ then do $f$ and $g$ have a
  common interlacing? This is true for $f\pposlace g$ where
  $f,g\in\polypos{1}$. 
\end{question}

\begin{question}
  If $\sum f_{ij}(x)y^iz^j\in\polypos{3}$ and $r$ is a positive integer then is the following
  determinant stable?
\[ 
\begin{vmatrix}
  f_{00} & f_{10} & \cdots & f_{r0} \\
  f_{01} & f_{11} & \cdots & f_{r1} \\
  \vdots & \vdots&        & \vdots \\
  f_{0r} & f_{1r} & \cdots & f_{rr}
\end{vmatrix}
\]
\end{question}


\end{document}